\documentclass{AIAA}
\usepackage{amssymb}
\usepackage[cmex10]{amsmath}
\usepackage{amsthm}

\newcommand{\tr}{\mathrm{tr}}

\newcommand{\diag}{\mathrm{diag}}

\renewcommand{\vec}[1]{\boldsymbol{#1}}

\renewcommand{\v}[1]{\vec{#1}}

\newcommand{\tv}[1]{\tilde{\vec{#1}}}
\newcommand{\dv}[1]{\dot{\vec{#1}}}
\newcommand{\hv}[1]{\hat{\vec{#1}}}
\newcommand{\cv}[1]{\check{\vec{#1}}}

\renewcommand{\dh}[1]{\dot{\hat{#1}}}
\newcommand{\dt}[1]{\dot{\tilde{#1}}}

\newcommand{\dhv}[1]{\dot{\hat{\vec{#1}}}}
\newcommand{\dtv}[1]{\dot{\tilde{\vec{#1}}}}

\newcommand{\Pa}{\mathbb{P}_a}
\newcommand{\Ps}{\mathbb{P}_s}
\newcommand{\inrt}{^{\mathcal{I}}}

\newcommand{\crs}{_{\times}}
\newcommand{\err}{_{\mathrm{err}}}
\newcommand{\rf}{_{\mathrm{ref}}}
\newcommand{\ddt}{\frac{\mathrm{d}}{\mathrm{d}t}}
\newcommand{\dddtdt}{\frac{\mathrm{d}^2}{\mathrm{d}t^2}}

\newtheorem{lem}{Lemma}
\newtheorem{thm}{Theorem}

\begin{document}

\title{Generalized Nonlinear Complementary Attitude~Filter}

\author{Kenneth J. Jensen\footnote{Control Systems Engineer, Makani Power Inc., 2175 Monarch St. Alameda, CA.}}
\affiliation{Makani Power Inc., Alameda, CA, 94501}

\maketitle

\section{Introduction}
The multiplicative extended Kalman filter (MEKF) is a proven
attitude estimation technique\cite{Leff82,Mark03}, that is the {\em
de facto} standard for low-cost inertial measurement units (IMUs).
However, the MEKF does have its drawbacks. It is computationally
expensive, difficult to tune, and potentially subject to divergence
due to numerical errors.  This has lead to the development of
alternate attitude estimation algorithms such as the unscented
Kalman filter \cite{Cras03,Juli04}, generalized MEKF \cite{Mart10},
invariant extended Kalman filter \cite{Bonn09b}, particle filters
\cite{Chen04,MaJi05}, and a host of nonlinear observers
\cite{Salc91,Thei03,Camp06,Bonn08,Maho08,Vasc08,Bonn09,Zama10}.
These techniques and others are summarized in a review article
\cite{Cras07} by Crassidis {\em et al.}.

A relatively new and popular estimation technique is Mahony's
nonlinear complementary filter \cite{Maho05,Maho08}.  This filter
boasts computational efficiency, a small number of easily tuned,
intuitive parameters, and a proof of almost global asymptotic
stability. Moreover, this filter has been shown to perform similarly
to the traditional MEKF \cite{Maho10}.

This work describes a new attitude estimation technique that is a
generalization of Mahony's nonlinear complementary filter.
Interestingly, this generalized attitude filter shares a close
mathematical relationship with the MEKF, and indeed both the MEKF
without gyro bias correction and the constant gain MEKF with gyro
bias correction may be viewed as special cases of the generalized
attitude filter.

\section{Attitude Estimation} \label{sec:attest}

The attitude estimation problem consists of combining measurements
from various potentially imperfect sensors located onboard an object
({\em e.g.} vehicle, aircraft, {\em etc}...) into an accurate
estimate of the attitude of the object. There are many versions of
this problem, which vary in parameters such as the attitude
representation, object dynamics, and measurement models.

The two primary attitude representations used in this work are the
direction cosine matrices for a global representation and the Euler
vector for small changes in attitude near the identity.  The global
attitude, meaning the rotation from the inertial frame to the frame
that rotates with the object ({\em i.e.} body frame), is denoted by
$C$.  From another viewpoint, $C$ represents the transformation from
body to inertial coordinates: $C \v{v} = \v{v}\inrt$.  (As described
in the Appendix, the notation $\inrt$ denotes vectors or matrices in
inertial coordinates; otherwise, vectors or matrices should be
assumed to be in body coordinates.) Small changes in attitude are
denoted by $\tilde{C}$ or by the Euler vector $\v{x}$ where
$\tilde{C} \approx I + \v{x}\crs$.

No assumptions are made as to the dynamics of the rotating object.
Rather, the process equations are based on the kinematics of the
rotation group itself.  Specifically, the attitude $C$ of an object
with angular velocity $\v{\omega}$, measured in the body frame,
evolves according to
\begin{equation} \label{eqn:kin}
\dot{C} = C \v{\omega}\crs
\end{equation}
The generality of the kinematics allows these results to apply to a
wide range of attitude estimation problems.

Finally, this work assumes a typical set of measurements, including an
angular rate measurement from a 3-axis rate gyro and a set of vector
measurements from, for example, a 3-axis magnetometer measuring
Earth's magnetic field and a 3-axis accelerometer measuring, with
appropriate filtering, Earth's gravity vector.  The angular rate
measurement, $\cv{\omega}$, is corrupted from the true angular rate,
$\v{\omega}$, by additive zero-mean white noise, $\v{\eta}_{\omega}$,
and a slowly varying bias error, $\v{b}$, which is driven by a white
noise process, $\v{\eta}_b$. Similarly, the vector measurements,
$\cv{v}_n$, of inertial vectors in the body frame are corrupted from
the true vectors, $C^T \v{v}\inrt_n$ by additive white noise,
$\v{\eta}_{v_n}$. Finally, it is assumed that the vector measurements
have all been normalized to be unit vectors.  To summarize, the
angular rate and vector measurement models used in this work are
\begin{subequations}
\label{eqn:measmod}
\begin{align}
\cv{\omega} &= \vec{\omega} + \vec{b}  + \v{\eta}_{\omega},\;\;\;\; \dv{b} = \v{\eta}_b \\
\cv{v}_n &= C^T \vec{v}\inrt_n +
\v{\eta}_{v_n},\;\;\;\;|\v{v}\inrt_n| = 1
\end{align}
\end{subequations}
The kinematics (\ref{eqn:kin}) and measurement model
(\ref{eqn:measmod}) completely specify the attitude estimation
problem.

\section{Multiplicative Extended Kalman Filter} \label{sec:mekf}

The MEKF applies the extended Kalman filter (EKF) formalism
\cite{Gelb74} to the attitude estimation problem. Although the MEKF
typically uses quaternions as its attitude representation, it is
reformulated here in terms of direction cosine matrices and Euler
vectors to facilitate comparison with the nonlinear complementary
filter.  After the derivation of the process equations for the Euler
vector, this description of the MEKF closely follows the form given
in \cite{Mark03}.

The true attitude $C$, estimated attitude $C\rf$, and error in
attitude $\tilde{C}$ are related according to
\begin{equation}
\tilde{C} = C\rf^T C = I + \v{x}\crs + \ldots \label{eqn:roterr}
\end{equation}
where $\tilde{C}$ is linearized about the identity with the
three-component Euler vector, $\v{x}$, as discussed earlier. The
process equations for the Euler vector are determined by the
kinematics of the rotation group (\ref{eqn:kin}).  These kinematics
apply equally well to the reference attitude: $\dot{C}\rf = C\rf
(\v{\omega}\rf)\crs$, where $\v{\omega}\rf$ is the angular velocity
of the reference attitude.  Thus,
\begin{align*}
\dt{C} &= \dot{C}\rf^T C + C\rf^T \dot{C} \\
&= -(\v{\omega}\rf)\crs \tilde{C} + \tilde{C} \v{\omega}\crs \\
&= (\v{\omega} - \v{\omega}\rf)\crs + \left[ \v{x}\crs, \frac{1}{2}(\v{\omega} + \v{\omega}\rf)\crs \right] \\
&\;\;\;\;+ \left\{ \v{x}\crs, \frac{1}{2}(\v{\omega} -
\v{\omega}\rf)\crs \right\} + \ldots
\end{align*}
The symmetric terms are second or higher order in $\v{x}\crs$.
Retaining only the first-order anti-symmetric terms of
(\ref{eqn:roterr}) yields the equation for the evolution of the
Euler vector:
\begin{equation}\label{eqn:procperfmeas}
\dv{x} \approx \v{\omega} - \v{\omega}\rf - \frac{1}{2}(\v{\omega} +
\v{\omega}\rf) \times \v{x}.
\end{equation}

It is now possible to apply the EKF formalism to form an estimate,
$\hv{x}$, of this attitude error vector.  Following the typical
procedure for the MEKF, $\v{\omega}\rf$ is set such that $\langle
\hv{x} \rangle =\v{0}$.  Thus, $\v{\omega}\rf$ represents an
estimate of the error between the current reference attitude $C\rf$
and the true attitude, $C$, which may now be estimated simply by
integrating (\ref{eqn:kin}) with $\v{\omega} = \v{\omega}\rf$.
Finally, similar to the results in \cite{Mark03}, the defining
equations for the continuous-time MEKF are:
\begin{subequations}
\label{eqn:mekf}
\begin{align}
\v{\omega}\rf &= \cv{\omega} - \hv{b} + P_a \sum_n \sigma^{-2}_{v_n} \left( \cv{v}_n \times {\v{v}\rf}_n \right) \label{eqn:mekf_a} \\
\dhv{b} &= P_c^T \sum_n \sigma^{-2}_{v_n} \left( \cv{v}_n \times
{\v{v}\rf}_n \right) \label{eqn:mekf_b}
\end{align}
\end{subequations}
where ${\v{v}\rf}_n = C^T\rf \v{v}\inrt_n$ and $\sigma_{v_n}$ is the
measurement noise covariance based on $\v{\eta}_{v_n}$. Also, the
covariance matrix, $P$, has been partitioned into $3\times 3$ blocks
\begin{equation*}
P =\left[\begin{array}{cc}
P_a & P_c\\
P_c^T & P_b%
\end{array}\right]
\end{equation*}
with $P_a$ representing the attitude covariances and $P_b$
representing the bias covariances. The covariance matrix evolves
according to the standard Ricatti equation, which becomes
\begin{align} \label{eqn:Pdyn}
\dot{P} =& \left[\begin{array}{cc}
[P_a, \v{\omega}'\crs] - 2 \Ps(P_c) & -\v{\omega}'\crs P_c - P_b\\
P_c^T \v{\omega}'\crs - P_b & 0\end{array}\right] \nonumber\\
&+ \left[\begin{array}{cc}
\sigma_{\omega}^2 I & 0\\
0 & \sigma_b^2 I\end{array}\right] \nonumber \\
&+ \sum_n \sigma_{v_n}^{-2} \left[\begin{array}{cc}
P_a ({\v{v}\rf}_n)^2\crs P_a & P_a ({\v{v}\rf}_n)^2\crs P_c \\
P_c^T ({\v{v}\rf}_n)^2\crs P_a & P_c^T ({\v{v}\rf}_n)^2\crs P_c \\
\end{array}\right]
\end{align}
where $\v{\omega}' = \frac{1}{2}(\cv{\omega} + \v{\omega}\rf -
\hv{b})$ and $\sigma_{\omega}^2$ and $\sigma_b^2$ are process noise
covariances based on $\v{\eta}_{\omega}$ and $\v{\eta}_b$
respectively.

For reasons discussed in Section \ref{sec:genfilt} below, the full
continuous-time MEKF is not a special case of the generalized
attitude filter.  However, two important special cases of the MEKF
are related to the generalized attitude filter.  These are the
bias-free MEKF and the constant gain MEKF.  The bias-free MEKF
simply assumes that the gyro measurement has no bias term,
$\hv{b}=\v{0}$.  The constant gain MEKF, which is often used to
reduce computational load, retains the gyro bias term but assumes
that the covariance matrix is constant at the value it would
approach as $t \rightarrow \infty$ with $\v{\omega}=\v{0}$.

\section{Generalized Nonlinear Complementary Attitude Filter}
\label{sec:genfilt}

The nonlinear complementary filters are a set of attitude estimators
inspired by traditional linear complementary filters.  Like a linear
complementary filter, the nonlinear complementary filters combine
low frequency attitude information from a set of vector measurements
with high frequency attitude information from rate gyros.  An
interesting feature of this family of estimators is the ability to
prove the almost global asymptotic stability of the estimator. This
work demonstrates how the stability proof may be extended to a
generalization of the nonlinear complementary filters. Moreover, it
is shown that the deterministic \cite{Reif99} forms of the MEKF
without gyro biases and the constant gain MEKF with gyro biases are
special cases of this generalized attitude filter, and thus, it is
proven that these forms of the MEKF are almost globally
asymptotically stable.

The form for Mahony's explicit complementary filter \cite{Maho08} is
\begin{subequations}
\label{eqn:mahncf}
\begin{alignat}{2}
\dh{C} &= \hat{C} \left( \cv{\omega} - \hv{b} + k_P \v{\omega}\err \right)\crs, \;\;\; &k_P>0 \\
\dhv{b} &= -k_I \v{\omega}\err, \;\;\; &k_I>0 \\
\v{\omega}\err &= \sum_n k_n \cv{v}_n \times \hv{v}_n, \;\;\; &k_n>0
\end{alignat}
\end{subequations}
where $\hv{v}_n \equiv \hat{C}^T \vec{v}\inrt_n$.  This filter may
be generalized simply by replacing the positive constant scalar
gains, $k_P$ and $k_I$, by potentially time-varying
positive-definite matrix gains, $K_P$ and $K_I$.
\begin{subequations}
\label{eqn:genfilt}
\begin{alignat}{2}
\dh{C} &= \hat{C} \left(  \cv{\omega} - \hv{b} + K_P \vec{\omega}\err \right)\crs,\;\;\; &K_P>0 \\
\dhv{b} &= -K_I \vec{\omega}\err, &K_I>0 \\
\vec{\omega}\err &= \sum_n k_n \cv{v}_n \times \hv{v}_n, &k_n>0
\end{alignat}
\end{subequations}
At this point, it is interesting to note that the $( \cv{\omega} -
\hv{b}) + K_P \vec{\omega}\err$ term from this filter is identical
to the $\vec{\omega}_{\mathrm{ref}}$ term (\ref{eqn:mekf_a}) from
the MEKF, including the positive-definite nature of the covariance
matrix. Moreover, the equation for $\dhv{b}$ here is similar to
(\ref{eqn:mekf_b}), the equation for the derivative of the bias
estimate in the MEKF. However, an important distinction is that the
integral matrix gain, $-P_c^T$, in the MEKF is not necessarily
positive definite.

It is easier to analyze the stability properties of the error
dynamics of the filter rather than the filter itself. Thus, the
following definitions for the errors between the true and estimated
attitudes and biases are introduced.
\begin{subequations}
\label{eqn:errdef}
\begin{align}
\tilde{C} &= \hat{C}^T C\\
\tv{b} &= \vec{b} - \hv{b}
\end{align}
\end{subequations}
The definition for the attitude error, $\tilde{C}$, here is the same
as the definition used in the MEKF.  By these definitions, the
filter has converged to true attitude and bias when $(\tilde{C},
\tv{b}) = (I, \v{0})$.

Combining the definitions for the error terms (\ref{eqn:errdef}),
the definition of the generalized attitude filter
(\ref{eqn:genfilt}), and the measurement models (\ref{eqn:measmod})
with no noise terms yields the equations for the error dynamics:
\begin{subequations}
\label{eqn:errdyn}
\begin{align}
\dt{C} &= [\tilde{C}, \vec{\omega}\crs] - (\tv{b} + K_P \vec{\omega}\err)\crs \tilde{C} \\
\dtv{b} &= K_I \vec{\omega}\err \\
\vec{\omega}\err &= \sum_n k_n \vec{v}_n \times \tilde{C} \vec{v}_n
\end{align}
\end{subequations}

As demonstrated in the proof below, the equilibria of the error
dynamics, denoted $\tilde{C}_*$, are determined by the inertial
vectors $\v{v}_n\inrt$ and the weights $k_n$, or more precisely by
the matrix $M$:
\begin{equation}
M = C^T M\inrt C \;\;\; \mathrm{with}\;\;\; M\inrt \equiv \sum_n k_n
\vec{v}\inrt_n (\vec{v}\inrt_n)^T. \label{eqn:Mdef}
\end{equation}
The stability proof relies on $M$ being positive semi-definite with
distinct eigenvalues, which as shown in \cite{Maho08}, is true if
there are at least two non-parallel measurement vectors.
Intuitively, the equilibria occur when the attitude error,
$\tilde{C}$, is the identity rotation ({\em i.e.} the filter has
converged) or a rotation of $\pi$ rad about one of the principle
axes of $M$. More specifically, the equilibria occur when
$\Pa(\tilde{C} M) = 0$, which according to Lemma \ref{lem:cstar}
below, implies that the equilibria are given by $\tilde{C}_{*i}
\equiv U D_i U^T$ where $D_0 = I$, $D_1 = \diag(1,-1,-1)$, $D_2 =
\diag(-1,1,-1)$, and $D_3 = \diag(-1,-1,1)$ for $M=U \Lambda U^T$
with diagonal $\Lambda$ and orthogonal $U$. The proof below first
demonstrates that the equilibria indeed occur at $\tilde{C}_{*i}$,
and then analyzes the stability properties of each equilibrium.

The following proof of the stability characteristics of the
generalized attitude filter is substantially similar to the proofs
given in \cite{Maho08}.  The primary difference is the extension of
the proof to handle potentially time-varying matrix gains rather
than constant scalar gains.  Some results from \cite{Maho08}, such
as the following lemma, transfer with no modification, and
consequently are simply restated here.  The stability properties of
the generalized attitude filter are now characterized with the
following theorems.

\begin{lem} \label{lem:cstar}
Suppose $\tilde{C} \in SO(3)$ and $M$ is positive semi-definite with
distinct eigenvalues and decomposition $M=U \Lambda U^T$ for
orthogonal $U$ and diagonal $\Lambda$. Then, $\Pa(\tilde{C} M) = 0$
if and only if $\tilde{C} = \tilde{C}_{*i} \equiv U D_i U^T$ where
$D_0 = I$, $D_1 = \diag(1,-1,-1)$, $D_2 = \diag(-1,1,-1)$, and $D_3
= \diag(-1,-1,1)$. \cite{Maho08}
\end{lem}

\begin{thm}[Stability of the generalized attitude filter] \label{thm:asymstab}
Consider the error dynamics described by (\ref{eqn:errdyn}). Suppose
that $K_P$, $K_I$, and $M$ are positive definite, $\dot{K}_I$ is
positive semi-definite, and $k_n>0$. Further, suppose that $K_P$ and
$K_I$ are upper and lower bounded by positive constants;
$\dot{K}_P$, $\dot{K}_I$, $\ddot{K}_I$, and $\vec{\omega}$ are
bounded; and $M$ has distinct eigenvalues. Then, the equilibrium
point $(\tilde{C}, \tv{b}) = (I, \vec{0})$ of the error dynamics is
asymptotically stable with a domain of attraction
$\mathcal{D}=\left\{ (\tilde{C}, \tv{b}) \in SO(3) \times
\mathbb{R}^3 - \{ (\tilde{C}_{*i}, \v{0})\, |\, i=1,2,3\} \right\}$
and is locally exponentially stable.
\end{thm}

\begin{proof}
Consider the Lyapunov function candidate
\begin{equation}
v = \sum_n k_n - \tr(\tilde{C} M) + \frac{1}{2} \tilde{\vec{b}}^T
K_I^{-1} \tilde{\vec{b}}
\end{equation}
Following the standard procedure, the time derivative of $v$ is
calculated:
\begin{align}
\dot{v} &= - \tr( \dt{C} M + \tilde{C} \dot{M} ) + \tv{b}^T K_I^{-1} \dtv{b} + \frac{1}{2} \tv{b}^T \ddt K_I^{-1} \tv{b} \nonumber \\
&= -\tr \left( [\tilde{C} M, \vec{\omega}\crs] - (\tv{b} + K_P \vec{\omega}\err)\crs \tilde{C} M \right) \nonumber\\
&\;\;\;\;+ \tv{b}^T \vec{\omega}\err  - \frac{1}{2} \tv{b}^T K_I^{-1} \dot{K}_I K_I^{-1} \tv{b} \nonumber \\
&= \tr \left( (\tv{b} + K_P \vec{\omega}\err)\crs \Pa(\tilde{C} M) \right) - \frac{1}{2} \tr \left( \tv{b}\crs (\vec{\omega}\err)\crs  \right) \nonumber\\
&\;\;\;\;- \frac{1}{2} \tv{b}^T  K_I^{-1} \dot{K}_I K_I^{-1} \tv{b} \nonumber \\
&= \tr \left( (K_P \vec{\omega}\err)\crs \Pa(\tilde{C} M) \right) -
\frac{1}{2} \tv{b}^T  K_I^{-1} \dot{K}_I K_I^{-1} \tv{b}
\label{eqn:vdotnotsimp}
\end{align}
where the identity (\ref{eqn:traceident}) from the Appendix and the
result $(\v{\omega}\err)\crs = 2\Pa(\tilde{C}M)$, also obtainable
from the identities (\ref{eqn:crossidents}), were used in the
simplification. As $\dot{K}_I \ge 0$, the second term in
(\ref{eqn:vdotnotsimp}) is negative semi-definite. To show that the
first term is also negative semi-definite, it is useful to rewrite
$K_P \vec{\omega}\err$ in terms of the measurement vectors
\begin{align}
(K_P \vec{\omega}\err)\crs &= \left( \sum_n k_n (Q \v{v}_n) \times (Q \tilde{C} \vec{v}_n) \right)\crs \nonumber\\
&= \sum_n k_n \left( Q \tilde{C}\vec{v}_n (\vec{v}_n)^T Q^T - Q \vec{v}_n (\tilde{C}\vec{v}_n)^T Q^T \right) \nonumber\\
&= 2 Q \Pa(\tilde{C} M ) Q^T \nonumber
\end{align}
where $Q = \pm \sqrt{\det K_P} K_P^{-1}$ is a positive or negative
definite matrix.  Then, using the fact that $Q$ has a decomposition
$Q=S^T \Lambda S$ for orthogonal $S$ and diagonal $\Lambda$:
\begin{align}
&\tr \left( (K_P \vec{\omega}\err)\crs \Pa(\tilde{C} M) \right) \nonumber\\
&\;\;\;\;= 2\, \tr \left( Q \Pa(\tilde{C} M) Q^T \Pa(\tilde{C} M) \right) \nonumber \\
&\;\;\;\;= 2\, \tr \left( (S^T \Lambda S) \Pa(\tilde{C} M) (S^T \Lambda S) \Pa(\tilde{C} M) \right) \nonumber \\
&\;\;\;\;= 2\, \tr \left( \Lambda (S \Pa(\tilde{C} M) S^T) \Lambda (S \Pa(\tilde{C} M) S^T) \right) \nonumber \\
&\;\;\;\;= 2\, \tr \left( (\Lambda \Pa( S \tilde{C} M S^T ))^2
\right) \le 0 \label{eqn:trle0}
\end{align}
This results in the trace of the square of the product of a diagonal
matrix and an antisymmetric matrix, which is easily shown to be
negative semi-definite.  Thus, combining (\ref{eqn:trle0}) and
(\ref{eqn:vdotnotsimp}), it is shown that $\dot{v} \le 0$.  Finally,
with the assumptions on the boundedness of $K_P$, $K_I$,
$\dot{K}_P$, $\dot{K}_I$, $\ddot{K}_I$, and $\vec{\omega}$, it is
straightforward to show that $\ddot{v}$, given by
\begin{align}
\ddot{v} =& 4\, \tr \left( Q \Pa(\tilde{C} M) \left(\dot{Q} \Pa(\tilde{C} M) \right.\right. \nonumber\\
& \left.\left. + Q \Pa( [\tilde{C} M, \vec{\omega}\crs] - 2Q \Pa(\tilde{C}M) Q^T \tilde{C} M) \right)\right) \nonumber\\
&+ \tv{b}^T \ddt K_I^{-1} \dtv{b}  + \frac{1}{2} \tv{b}^T \dddtdt
K_I^{-1} \tv{b} \nonumber
\end{align}
is bounded.  Therefore, Barbalat's lemma \cite{Khal96} implies that
$\dot{v}$ and thus $\Pa(\tilde{C} M)$ tend asymptotically to zero.
According to Lemma \ref{lem:cstar}, the attitude error, $\tilde{C}$
must approach one of the equilibria $\tilde{C}_{*i}$.  Moreover,
$\Pa(\tilde{C} M) \rightarrow 0$ implies that $\v{\omega}\err
\rightarrow 0$, and thus, after substituting the relation
$\dt{C}_{*i} = [\tilde{C}_{*i}, \v{\omega}\crs]$ into
(\ref{eqn:errdyn}), it is shown that $\tv{b}$ must also tend
asymptotically to zero.

To demonstrate the stability characteristics of the various
equilibria, the system is linearized about each $\tilde{C}_{*i}$.
Because the equilibria attitude errors are constant in the inertial
frame ({\em i.e.} $\dt{C}_{*i}\inrt = 0$), it is easiest to conduct
the linearization in the inertial frame. The error dynamics
(\ref{eqn:errdyn}) expressed in the inertial frame are
\begin{subequations}
\label{eqn:errdyninrt}
\begin{align}
\dt{C}\inrt &= -(\tv{b}\inrt + K_P\inrt \vec{\omega}\err\inrt)\crs \tilde{C}\inrt \\
\dtv{b}\inrt &= \v{\omega}\crs\inrt \tv{b}\inrt + K_I\inrt
\v{\omega}\err\inrt
\end{align}
\end{subequations}
where $\tilde{C}\inrt = C \tilde{C} C^T$ as described in the
Appendix. Let $\tilde{C}\inrt \approx \tilde{C}_{*i}\inrt (1 +
\v{x}\crs\inrt)$ and $\tv{b}\inrt \approx -\v{y}\inrt$.  First, the
measurement error vector is linearized:
\begin{align}
\v{\omega}\err\inrt &= \sum_n k_n \v{v}_n\inrt \times \tilde{C}\inrt \v{v}_n\inrt \nonumber \\
&\approx \sum_n k_n \v{v}_n\inrt \times \tilde{C}_{*i}\inrt (1+\v{x}\inrt\crs) \v{v}_n\inrt \nonumber \\
&= \sum_n k_n \v{v}_n\inrt \times \tilde{C}_{*i}\inrt \v{v}_n\inrt - \sum_n k_n \v{v}_n\inrt \times \tilde{C}_{*i}\inrt (\v{v}_n\inrt)\crs \v{x}\inrt \nonumber \\
&= - \left( \sum_n k_n (\tilde{C}_{*i}\inrt \v{v}_n\inrt)\crs (\v{v}_n\inrt)\crs \right) \v{x}\inrt \nonumber \\
&= \tilde{C}_{*i}\inrt \sum_n k_n \left( (\v{v}_n\inrt)^T \tilde{C}_{*i}\inrt \v{v}_n\inrt I - \v{v}_n\inrt (\v{v}_n\inrt)^T \tilde{C}_{*i}\inrt \right) \v{x}\inrt \nonumber \\
&= \tilde{C}_{*i}\inrt \left( \tr(M\inrt) I - M\inrt \right) \tilde{C}_{*i}\inrt \v{x}\inrt \nonumber \\
&= A_i\inrt \v{x}_i\inrt \label{eqn:werrlin}
\end{align}
where $A_i\inrt=\tilde{C}_{*i}\inrt \left( \tr(M\inrt) I - M\inrt
\right)$ and $\v{x}_i\inrt = \tilde{C}_{*i}\inrt \v{x}\inrt$.
Substituting (\ref{eqn:werrlin}) into (\ref{eqn:errdyninrt}) with
the linearized $\tilde{C}\inrt$ and $\tv{b}\inrt$ results in
\begin{equation}
\left[\begin{array}{c}
\dv{x}_i\inrt\\
\dv{y}\inrt\end{array}\right] = \left[\begin{array}{cc}
-K_P\inrt A_i\inrt  & I\\
-K_I\inrt A_i\inrt  & \v{\omega}\inrt\crs%
\end{array}\right]
\left[\begin{array}{c}
\v{x}_i\inrt\\
\v{y}\inrt%
\end{array}\right].
\label{eqn:lindyn}
\end{equation}

To demonstrate the instability of the equilibria $\tilde{C}_{*i}$
for $i=1,2,3$, it is necessary to show that a trajectory starting
arbitrarily close to $\tilde{C}_{*i}$ ({\em i.e.} $|\v{\xi}\inrt|$
arbitrarily close to zero with $\v{\xi}\inrt=(\v{x}\inrt_i,
\v{y}\inrt)$) must eventually diverge from the compact set, which
contains the equilibrium, defined by $|\v{\xi}\inrt| \le r$ for some
$r$ chosen such that the linearization is still valid. Consider the
cost function
\begin{equation}
s = \frac{1}{2} (\v{x}_i\inrt)^T A_i\inrt \v{x}_i\inrt + \frac{1}{2}
(\v{y}\inrt)^T (K_I\inrt)^{-1} \v{y}\inrt
\end{equation}
The time derivative of $s$ is given by
\begin{align}
\dot{s} =& \frac{1}{2} \bigg( -(\v{x}_i\inrt)^T (A_i\inrt)^T (K_P\inrt)^T A_i\inrt \v{x}_i\inrt + (\v{y}\inrt)^T A_i\inrt \v{x}_i\inrt \nonumber\\
& - (\v{x}_i\inrt)^T A_i\inrt K_P\inrt A_i\inrt \v{x}_i\inrt + (\v{x}_i\inrt)^T A_i\inrt \v{y}\inrt - (\v{y}\inrt)^T \v{\omega}\inrt\crs (K_I\inrt)^{-1} \v{y}\inrt \nonumber \\
& - (\v{x}_i\inrt)^T (A_i\inrt)^T (K_I\inrt)^T (K_I\inrt)^{-1} \v{y}\inrt + (\v{y}\inrt)^T (K_I\inrt)^{-1} \v{\omega}\inrt\crs \v{y}\inrt \nonumber \\
& - (\v{y}\inrt)^T (K_I\inrt)^{-1} K_I\inrt A_i\inrt \v{x}_i\inrt + (\v{x}_i\inrt)^T \dot{A}_i\inrt \v{x}_i\inrt \nonumber \\
& + (\v{y}\inrt)^T \ddt (K_I\inrt)^{-1} \v{y}\inrt \bigg) \nonumber \\
=& -(\v{x}_i\inrt)^T (A_i\inrt)^T K_P\inrt A_i\inrt \v{x}_i\inrt \nonumber \\
& - \frac{1}{2} (\v{y}\inrt)^T (K_I\inrt)^{-1} (\dot{K}_I)\inrt
(K_I\inrt)^{-1} \v{y}\inrt
\end{align}
which is negative definite as $\dot{K}_I$ and thus
$(\dot{K}_I)\inrt$ are positive definite. For $i=1,2,3$, $A_i$ has
at least one negative eigenvalue, and thus there is some
$\v{\xi}\inrt_0$, with magnitude arbitrarily close to zero, for
which $s(\v{\xi}\inrt_0)<0$. As $\dot{s}<0$ and because $r$ was
chosen such that the linearization is valid, trajectories
$\v{\xi}\inrt(t)$ starting at $\v{\xi}\inrt_0$ must eventually pass
through the sphere with radius $r$. Thus, these equilibria are
unstable.

The local exponential stability of the equilibrium point
$(\tilde{C}, \tv{b}) = (I, \v{0})$ is now proven.  Take the
$\tilde{C}_{*0} = I$ case of the linearized system
(\ref{eqn:lindyn}) and return the system to body coordinates.  The
simplified linearized system is
\begin{equation} \label{eqn:simplin}
\dv{\xi} = B \v{\xi},\;\;\;\; B = \left[\begin{array}{cc}
-K_P A_0 - \v{\omega}\crs & I\\
-K_I A_0 & 0%
\end{array}\right]
\end{equation}
where  $A_0 = \tr(M)I - M$ is positive definite.  To prove the
exponential stability of the equilibrium point $\vec{\xi} =
(\vec{0}, \vec{0})$ of the linearized system (\ref{eqn:simplin}),
consider the Lyapunov function candidate
\begin{equation}
w = \frac{1}{2} \v{\xi}^T P \v{\xi},\;\;\;\; P =
\left[\begin{array}{cc}
A_0 & -\alpha A_0\\
-\alpha A_0 & K_I^{-1}%
\end{array}\right]
\end{equation}
where $\alpha$ is chosen such that $P$ is positive definite, or more
specifically such that $\alpha^2 A_0 < K_I^{-1}$. The derivative of
$w$ along the trajectories of the system is
\begin{equation}
\dot{w} = \v{\xi}^T \left( P B + B^T P + \dot{P} \right) \v{\xi} =
-\v{\xi}^T Q \v{\xi} \nonumber
\end{equation}
where
\begin{align}
Q &= -(P B + B^T P + \dot{P}) \nonumber \\
&= \left[\begin{array}{cc}
2 A_0 (K_P - \alpha K_I) A_0  &  -\alpha (A_0 K_P A_0 - A_0 \v{\omega}\crs )\\
-\alpha (A_0 K_P A_0 + \v{\omega}\crs A_0 ) & 2 \alpha A_0 + K_I^{-1} \dot{K}_I K_I^{-1}%
\end{array}\right]\nonumber
\end{align}
Using the Schur complement condition for positive definiteness, $Q$
is positive definite exactly when $K_P > \alpha K_I$ and
\begin{align}
2 \alpha A_0 &+ K_I^{-1} \dot{K}_I K_I^{-1} \nonumber\\
&> \frac{1}{2} \alpha^2 (A_0 K_P + \v{\omega}\crs)(K_P-\alpha
K_I)^{-1} (K_P A_0 - \v{\omega}\crs) \nonumber
\end{align}
With $\dot{K}_I \ge 0$, it is straightforward to show that there
exists some $\alpha>0$ such that $P>0$ and $Q>0$. Therefore,
$\dot{w}$ is upper bounded by a negative constant and the linearized
system is exponentially stable.

Together, the results on the asymptotic convergence of $(\tilde{C},
\tv{b})$ to $(\tilde{C}_{*i}, \v{0})$, the instability of the
$\tilde{C}_{*1}$, $\tilde{C}_{*2}$, and $\tilde{C}_{*3}$ equilibria,
and the stability of the $\tilde{C}_{*0}$ equilibrium show the
asymptotic stability of $(I, \v{0})$ with domain of attraction
$\mathcal{D}$.
\end{proof}

A direct corollary of Theorem \ref{thm:asymstab} is that the
deterministic \cite{Reif99} bias-free MEKF is almost globally
asymptotically stable. Simply identify $K_P \equiv P_a$ and set the
bias error in the proof to zero.  It is easy to show, given the
boundedness of $\v{\omega}$, that $P_a$ and $\dot{P}_a$ meet the
boundedness requirements of $K_P$ and $\dot{K}_P$ for the theorem.
The proof then proceeds without modification.  Another interesting
special case is that of the deterministic constant gain MEKF.  The
following theorem demonstrates that the deterministic constant gain
MEKF is almost globally asymptotically stable as well.

\begin{thm}
The deterministic constant gain MEKF defined by (\ref{eqn:mekf})
with $P_a = P_a(\infty)$ and $P_c = P_c(\infty)$ is asymptotically
stable with a domain of attraction $\mathcal{D}$ defined in Theorem
\ref{thm:asymstab}.
\end{thm}

\begin{proof}
The proof proceeds by demonstrating that the constant gain MEKF is a
special case of the generalized attitude filter.  It has already
been shown that the forms for the MEKF and the generalized attitude
filters are similar.  All that remains to be shown is the positive
definite nature of the matrix gains, $P_a$ and $-P_c$ in the
constant gain MEKF.

With $\v{\omega}=\v{0}$ and assuming the covariance matrix converges
then equations for the matrix gains are
\begin{subequations}
\begin{alignat}{2}
\label{eqn:Pa0}
\dot{P}_a &= -2 \Ps(P_c) + \sigma_{\omega}^2 I - P_a A_0 P_a &= 0\\
\label{eqn:Pb0}
\dot{P}_b &= \sigma_b^2 I - P_c^T A_0 P_c &= 0\\
\label{eqn:Pc0} \dot{P}_c^T &= -P_b - P_c^T A_0 P_a &= 0
\end{alignat}
\end{subequations}
where as in Theorem \ref{thm:asymstab}, $A_0 = \tr(M)I - M$. The
combining of (\ref{eqn:Pb0}) and (\ref{eqn:Pc0}) yields
\begin{equation} \label{eqn:PcPaPb}
P_c = -\sigma_b^2 P_a P_b^{-1}.
\end{equation}
Substitution into (\ref{eqn:Pa0}) results in
\begin{equation}
P_a P_b^{-1} + P_b^{-1} P_a = -\sigma_b^{-2} \sigma_{\omega}^2 I +
P_b^2 \nonumber
\end{equation}
from which it is evident that $P_a$ and $P_b$ are simultaneously
diagonalizable.  Thus, as $P_a$ and $P_b$ are positive definite
since the entire covariance matrix is positive definite,
(\ref{eqn:PcPaPb}) shows that $P_c$ is negative definite.

Identifying $K_P \equiv P_a$ and $K_I \equiv -P_c$ results in the
equations for the generalized attitude filter, which by Theorem
\ref{thm:asymstab} is asymptotically stable with a domain of
attraction $\mathcal{D}$.
\end{proof}

\section{Conclusions}
\label{sec:conclusion}

This paper introduces a very general class of attitude estimators
that contains Mahony's explicit nonlinear complementary filter, the
bias-free multiplicative extended Kalman filter (MEKF), and the
constant-gain MEKF as special cases. This generalized attitude
filter is a modification of Mahony's filter that simply replaces the
constant scalar gains by potentially time-varying matrix gains.  The
stability proofs developed for Mahony's filters extend naturally to
the generalized filter.  Thus, it is possible to prove the almost
global asymptotic stability of this general class of filters, and
consequently provide proof of the almost global asymptotic stability
of a few special cases of the MEKF. This generalized attitude filter
gives the filter designer an enormous space to tune and optimize the
filter, while ensuring stability.  It is important to note that this
generalized attitude filter does not include as a special case the
full MEKF, which may not necessarily have a positive definite
integral gain.

\section*{Appendix}

There is an unfortunate abundance of notation employed in this
paper, which results both from the nature of the work as well as an
attempt to display the results from the MEKF papers and the
nonlinear complementary paper in a compatible fashion.  The notation
most closely follows that in Mahony's paper \cite{Maho08}, with only
minor changes made to avoid conflicts with the other works. Table
\ref{tbl:notation} describes the major elements of the notation.

\begin{table}[!t]
\caption{Notation} \label{tbl:notation} \centering
\begin{tabular}{ll}
\hline
Notation & Interpretation \\
\hline
$a$ & True value \\
$\hat{a}$ & Estimated value \\
$\check{a}$ & Measured value \\
$\tilde{a}$ & Error between estimated and true value\\
$\langle a \rangle$ & Time-averaged expectation value\\
$A$ & Capital letters indicate matrices\\
$\tilde{A}$ & Left multiply error: $\tilde{A} A^T = \hat{A}^T$\\
$A_*$ & Equilibrium point \\
$A\inrt$ & Matrix in inertial frame: $A\inrt = C A C^T$\\
$\Ps(A)$, $\Pa(A)$ & Symmetric or anti-symmetric projection of $A$\\
$[A,B]$, $\{A,B\}$ & Commutator or anti-commutator of $A$ and $B$\\
$\v{x}$ & Bold lowercase letters indicate vectors\\
$\v{x}\inrt$ & Vector in inertial frame: $\v{x}\inrt = C \v{x}$\\
$\v{x}\crs$ & Matrix equivalent form of the cross product\\
\hline
\end{tabular}
\end{table}

This work also makes heavy use of a few identities and definitions
described below. The matrix equivalent form of the cross product,
also known as the cross operator, is defined by
\begin{equation*}
\vec{x}_{\times} \equiv \left[\begin{array}{ccc}
0 & -x_3 & x_2\\
x_3 & 0 & -x_1\\
-x_2 & x_1 & 0%
\end{array}\right].
\end{equation*}
The symmetric and anti-symmetric matrix operators are defined as
\begin{align*}
\Ps(A) &\equiv \frac{1}{2}(A + A^T)\\
\Pa(A) &\equiv \frac{1}{2}(A - A^T).
\end{align*}
The commutator and anti-commutator are defined as
\begin{align*}
[A,B] &\equiv A B - B A\\
\{A,B\} &\equiv A B + B A.
\end{align*}
Some useful identities for manipulating cross products and cross
operators are
\begin{subequations}
\begin{align}
(\v{x} \times \v{y})_{\times} &= [\v{x}\crs, \v{y}\crs] \label{eqn:crosscrossident}\\
Q \v{x} \times Q \v{y} &= (\det Q) Q^{-1} (\v{x} \times \v{y}) \label{eqn:matcross}\\
\v{x}^T \v{y} &= -\frac{1}{2} \tr \left( \v{x}\crs \v{y}\crs \right) \label{eqn:traceident}\\
\v{x}\crs \v{y}\crs &= \v{y}\v{x}^T - \v{y}^T\v{x} I
\end{align}
\label{eqn:crossidents}
\end{subequations}
where $Q$ in (\ref{eqn:matcross}) must be positive or negative
definite. The following identities are useful for expanding or
contracting matrices with the cross operator:
\begin{subequations}
\begin{align}
\Pa(A \v{x}\crs B) &= \frac{1}{2} (C \v{x})\crs \label{eqn:crident1}\\
S \v{x}\crs S^T &= (S \v{x})\crs \label{eqn:crident2}\\
A \v{x}\crs A &= (D \v{x})\crs \label{eqn:crident3}
\end{align}
\end{subequations}
where $A = S \Lambda_A S^T$, $B = S \Lambda_B S^T$, $C = S \Lambda_C
S^T$, and $D = S \Lambda_D S^T$ with $S \in SO(3)$ and where
\begin{align}
\Lambda_A &= \diag(\lambda_1,\lambda_2,\lambda_3) \nonumber \\
\Lambda_B &= \diag(\lambda'_1, \lambda'_2, \lambda'_3) \nonumber \\
\Lambda_C &= \diag(\lambda_2\lambda'_3 + \lambda'_2\lambda_3, \lambda_1\lambda'_3 +\lambda'_1\lambda_3, \lambda_1\lambda'_2 + \lambda'_1\lambda_2) \nonumber \\
\Lambda_D &= \diag(\lambda_2 \lambda_3, \lambda_1 \lambda_3,
\lambda_1 \lambda_2).\nonumber
\end{align}
Identities (\ref{eqn:crident2}) and (\ref{eqn:crident3}) are just
special cases of the first identity (\ref{eqn:crident1}).

\section*{Acknowledgments}
The author would like to thank Erik Chubb, Dr. Paula Echeverri,
Meagan Moore, Gabriel Popkin, Damon Vander Lind, and Dr. Corwin
Hardham for discussions on the generalized attitude filter and for
comments on the paper.

\section*{References}
\bibliographystyle{aiaa}
\bibliography{attitude}
\end{document}